\documentclass[11pt]{amsart}

\DeclareMathOperator{\Lie}{Lie}

\DeclareMathOperator{\Char}{char}
\usepackage[usenames,dvipsnames]{color}

\usepackage{amsmath,amssymb,amsthm,graphics,amscd,mathrsfs}
\usepackage{longtable}
\usepackage{cite}
\usepackage{color}
\usepackage{graphicx}
\usepackage{epsf}
\usepackage{fancyhdr,hyperref}
\usepackage{lscape}
\hypersetup{colorlinks=true,citecolor=Purple}

  \renewenvironment{thebibliography}[1]{
    \begin{oldthebibliography}{#1}
      \setlength{\parskip}{0ex}
      \setlength{\itemsep}{0ex}
  }
  {
    \end{oldthebibliography}
  }
  
\setlongtables
\begin{document}

\newcounter{rownum}
\setcounter{rownum}{0}
\newcommand{\ab}{\addtocounter{rownum}{1}\arabic{rownum}}

\newcommand{\x}{$\times$}
\newcommand{\bb}{\mathbf}

\newcommand{\Ind}{\mathrm{Ind}}
\newcommand{\R}{\mathrm{R}}
\newcommand{\RR}{\mathscr{R}}
\newcommand{\G}{\mathscr{G}}
\newcommand{\hra}{\hookrightarrow}
\newcommand{\sss}{\mathrm{ss}}
\newtheorem{lemma}{Lemma}[section]
\newtheorem{theorem}[lemma]{Theorem}
\newtheorem*{theorem*}{Theorem}
\newtheorem*{TA}{Main Theorem}
\newtheorem*{TB}{Theorem B}
\newtheorem*{TC}{Theorem C}
\newtheorem*{CorC}{Corollary C}
\newtheorem*{TD}{Theorem D}
\newtheorem*{TE}{Theorem E}
\newtheorem*{PF}{Proposition E}
\newtheorem*{C3}{Corollary 3}
\newtheorem*{T4}{Theorem 4}
\newtheorem*{C5}{Corollary 5}
\newtheorem*{C6}{Corollary 6}
\newtheorem*{C7}{Corollary 7}
\newtheorem*{C8}{Corollary 8}
\newtheorem*{claim}{Claim}
\newtheorem{cor}[lemma]{Corollary}
\newtheorem{conjecture}[lemma]{Conjecture}
\newtheorem{prop}[lemma]{Proposition}
\newtheorem{question}[lemma]{Question}
\theoremstyle{definition}
\newtheorem{example}[lemma]{Example}
\newtheorem{examples}[lemma]{Examples}
\newtheorem{algorithm}[lemma]{Algorithm}
\newtheorem*{algorithm*}{Algorithm}
\theoremstyle{remark}
\newtheorem{remark}[lemma]{Remark}
\newtheorem{remarks}[lemma]{Remarks}
\newtheorem{obs}[lemma]{Observation}
\theoremstyle{definition}
\newtheorem{defn}[lemma]{Definition}

  \def\hal{\unskip\nobreak\hfil\penalty50\hskip10pt\hbox{}\nobreak
  \hfill\vrule height 5pt width 6pt depth 1pt\par\vskip 2mm}

\renewcommand{\labelenumi}{(\roman{enumi})}
\newcommand{\Hom}{\mathrm{Hom}}
\newcommand{\Int}{\mathrm{int}}
\newcommand{\Ext}{\mathrm{Ext}}
\newcommand{\opH}{\mathrm{H}}
\newcommand{\D}{\mathscr{D}}
\newcommand{\soc}{\mathrm{Soc}}
\newcommand{\SO}{\mathrm{SO}}
\renewcommand{\O}{\mathrm{O}}
\newcommand{\Sp}{\mathrm{Sp}}
\newcommand{\SL}{\mathrm{SL}}
\newcommand{\GL}{\mathrm{GL}}
\newcommand{\PGL}{\mathrm{PGL}}
\newcommand{\OO}{\mathcal{O}}
\newcommand{\Y}{\mathbf{Y}}
\newcommand{\X}{\mathbf{X}}
\newcommand{\diag}{\mathrm{diag}}
\newcommand{\End}{\mathrm{End}}
\newcommand{\tr}{\mathrm{tr}}
\newcommand{\Stab}{\mathrm{Stab}}
\newcommand{\red}{\mathrm{red}}
\newcommand{\Aut}{\mathrm{Aut}}
\renewcommand{\H}{\mathcal{H}}
\renewcommand{\u}{\mathfrak{u}}
\newcommand{\Ad}{\mathrm{Ad}}
\newcommand{\N}{\mathcal{N}}
\newcommand{\Z}{\mathbb{Z}}
\newcommand{\la}{\langle}\newcommand{\ra}{\rangle}
\newcommand{\gl}{\mathfrak{gl}}
\newcommand{\g}{\mathfrak{g}}
\newcommand{\F}{\mathbb{F}}
\newcommand{\m}{\mathfrak{m}}
\renewcommand{\b}{\mathfrak{b}}
\newcommand{\p}{\mathfrak{p}}
\newcommand{\q}{\mathfrak{q}}
\renewcommand{\l}{\mathfrak{l}}
\newcommand{\del}{\partial}
\newcommand{\h}{\mathfrak{h}}
\renewcommand{\t}{\mathfrak{t}}
\renewcommand{\k}{\mathfrak{k}}
\newcommand{\Gm}{\mathbb{G}_m}
\renewcommand{\c}{\mathfrak{c}}
\renewcommand{\r}{\mathfrak{r}}
\newcommand{\n}{\mathfrak{n}}
\newcommand{\s}{\mathfrak{s}}
\newcommand{\Q}{\mathbb{Q}}
\newcommand{\z}{\mathfrak{z}}
\newcommand{\pso}{\mathfrak{pso}}
\newcommand{\so}{\mathfrak{so}}
\renewcommand{\sl}{\mathfrak{sl}}
\newcommand{\psl}{\mathfrak{psl}}
\renewcommand{\sp}{\mathfrak{sp}}
\newcommand{\Ga}{\mathbb{G}_a}

\newenvironment{changemargin}[1]{%
  \begin{list}{}{%
    \setlength{\topsep}{0pt}%
    \setlength{\topmargin}{#1}%
    \setlength{\listparindent}{\parindent}%
    \setlength{\itemindent}{\parindent}%
    \setlength{\parsep}{\parskip}%
  }%
  \item[]}{\end{list}}

\parindent=0pt
\addtolength{\parskip}{0.5\baselineskip}

\title{Levi decomposition of nilpotent centralisers in classical groups}

\author{Alex P. Babinski}
\address{Department of Mathematics, Statistics, and Computer Science, University of Wisconsin - Stout,
Menomonie, WI 54751, USA}
\email{babinskia@uwstout.edu}
\author{David I. Stewart}
\address{School of Mathematics and Statistics,
Herschel Building,
Newcastle,
NE1 7RU, UK}
\email{david.stewart@ncl.ac.uk}

\pagestyle{plain}
\begin{abstract}We check that the connected centralisers of nilpotent elements in the orthogonal and symplectic groups have Levi decompositions in even characteristic. This provides a justification for the identification of the isomorphism classes of the reductive quotients as stated in [Liebeck, Seitz; Unipotent and Nilpotent Classes in Simple Algebraic Groups and Lie Algebras]. \end{abstract}
\maketitle

\section{Introduction}

Let $G$ be a linear algebraic group over an arbitrary field $k$ with unipotent radical $U:=\RR_u(G)$. Then $U$ is by definition a subgroup of $G_{\bar k}$, where $G_{\bar k}$ is the base change of $G$ to the  algebraic closure $\bar k$ of $k$. In fact, the subgroup $U$ is defined to be the largest smooth, connected, unipotent normal subgroup of $G_{\bar k}$. We say $G$ has a \emph{Levi subgroup} $L$ if $G_{\bar k}=L_{\bar k}U$ and $L_{\bar k}\cap U=\{1\}$, scheme-theoretically; that is to say, that the following conditions hold:
\begin{align} &L_{\bar k}(\bar k)\cap U(\bar k)=\{1\}; \label{pts}\\
&\Lie(L_{\bar k})\cap \Lie(U)=0.\label{lie}\end{align}

The existence (or otherwise) of Levi subgroups is a central issue to address in understanding the subgroup structure of linear algebraic groups. When $k$ is a field of characteristic $0$, it is an old theorem of G.~D.~Mostow \cite{Mos56} that all linear algebraic groups have Levi subgroups. Essentially, the proof relies on Lie's theorem and exponentiation, both of which fail over fields of characteristic $p>0$. Indeed, algebraic groups need not have Levi subgroups over such fields. The points $G(W_2(k))$ of a reductive $k$-group $G$ over the length $2$ Witt vectors $W_2(k)$ furnish an example of such an algebraic group; see \cite[\S A.6]{CGP10} for a full account. (Also note that a minimal dimensional faithful representation for $G=\SL_2(W_2(k))$ is constructed in \cite{McN03}.) In this case one has a short exact sequence $1\to \g^{[1]}\to G(W_2(k))\to G\to 1$, where $\g=\Lie(G)$ and $\g^{[1]}$ is its first Frobenius twist as a $G$-module. Then the (unipotent) vector subgroup $\g^{[1]}\subseteq G(W_2(k))$ coincides with the unipotent radical of the latter. One can see that this sequence corresponds to an element of the rational (Hochschild) cohomology group $\opH^2(G,\g^{[1]})$ and indeed one has a suite of examples of $G$-modules $V$ where $\opH^2(G,V)\neq 0$ each giving rise to a non-split extension of $V$ by $G$ such that $V$ is the unipotent radical of the extension $E$ with no Levi factor. By contrast, if $G$ is a \emph{connected} linear algebraic group over $k$ with unipotent radical $U$ which is defined over $k$ then \cite[Thm.~B]{McN14} (see also \cite[Thm.~3.3.5]{SteUni}) shows that one can find a filtration of $U$ such that the sections have the structure of modules for $G/U$, and \cite{McN10} points out that the vanishing of the second Hochschild cohomology of these modules is enough to guarantee a Levi subgroup.

Certain interesting situations arise over an imperfect field $k$ since it is  possible that the unipotent radical $U$ may fail to be defined over $k$. This can happen in particular when one considers the case that $G$ is a pseudo-reductive group. The main result of the monograph \cite{CGP10} asserts that most \emph{pseudo-reductive} groups arise from Weil restriction of a reductive group across an inseparable extension of $k$. Moreover, if $G'$ is a reductive group that happens to be defined over $k$ and $k'/k$ is an inseparable extension, then the Weil restriction $\R_{k'/k}(G'_{k'})$ is a non-reductive linear algebraic group $G$ whose unipotent radical $U$ is not defined over $k$ but which contains a canonical copy of $G'$ as a Levi subgroup. For a general result on the existence of Levi subgroups in pseudo-reductive groups, see \cite[Thm.~3.4.6]{CGP10}.

In \cite[Prop.~5.10]{Jan04}, Jantzen shows, using arguments from \cite{Ric67} that when the characteristic $p$ of $k$ is good for $G$ the (smooth) centraliser $C_G(e)$ of a nilpotent elements $e\in\Lie(G)$ for $G$, a reductive group always has a Levi subgroup. In bad characteristic, this can apparently fail in the exceptional groups (see \cite[p283]{LS12}), though we have not found explicit examples. However, in this short note we wish to make the observation:

\begin{theorem*}\label{thethm}Let $G$ be a simple algebraic group of classical type over $k=\bar k$ of characteristic $2$ and $e\in\Lie(G)$ a nilpotent element. Then $C_G(e)^\circ_{\red}$ has a Levi decomposition.\end{theorem*}

(The centralisers of nilpotent elements in bad characteristic need not be smooth; the group $C_G(e)^\circ_{\red}$ is the unique smooth group whose $k$-points are the same as that of $C_G(e)^\circ$, hence is the centraliser in the sense of \cite{Spr98}.) Most of our work is done by \cite{LS12}, which finds a subgroup $L$ of $C_G(e)^\circ$ satisfying (\ref{pts}) above. It remains to show that (\ref{lie}) holds. Chasing through the proof of \cite[Prop.~5.11]{LS12} and applying a result of Vasiu we show this is the case. 

Having established the existence of a subgroup $L$ satisfying (\ref{pts}), the authors of \cite{LS12} do not appear to have made an attempt to justify their statement in \cite[Thm.~5.6]{LS12} that there is an isomorphism $C_G(e)^\circ_{\red}/\RR_u(C_G(e)^\circ_{\red})\cong L$ as algebraic groups and indeed this map can fail to be an isomorphism of algebraic groups, precisely when (\ref{lie}) does not hold. Hence our theorem provides the missing justification.

\section{Proof of the theorem}
In this section $k$ will denote an algebraically closed field of characteristic $2$.

The following is a brief version of \cite[Thm.~5.6]{LS12}. As explained in the introduction, the proof in \emph{op.~cit.} only establishes the isomorphisms at the level of the abstract groups of points.
\begin{theorem}Let $e$ be a nilpotent element of $\Lie(G)$ where $G=\Sp(V)$ or $\O(V)$ and $V$ is a vector space over $k$. Then there are integers $m_i$ and $a_i$ such that:
\begin{enumerate} \item If $G=\Sp(V)$, then $C_G(e)^\circ_{\red}/\RR_u(C_G(e))\cong \prod_i\Sp_{2a_i}$.
\item If $G=\O(V)$ then $C_G(e)^\circ_{\red}/\RR_u(C_G(e))\cong \prod_{m_i}\Sp_{2a_i}\times \prod_{m_i}I_{a_i}$, where $I_{a_i}$ is either $\SO_{2a_i}$ or $\SO_{2a_i+1}$.\end{enumerate}
\end{theorem}
A technical condition related to the action of $e$ on $V$ determines the integers $a_i$ and $m_i$ and the condition by which one decides the isomorphism class of $I_{a_i}$. Then \cite[Prop.~5.11]{LS12} finds subgroups $C$ such that $C_G(e)_{\red}=C\RR_u(C_G(e))$.

To prove our theorem, we use \cite[Thm.~1.2]{Vas05}. Recall that for a field $k$ of characteristic $p$, $\alpha_p$ denotes the height $1$ group scheme whose representing Hopf algebra is $k[X]/(X^p)$, the comultiplication being determined by $\Delta(X)=1\otimes X+X\otimes 1$. (It is also the first Frobenius kernel of the smooth additive group $\Ga$.) For us, \emph{loc.~cit.} takes the form:

\begin{theorem}[Vasiu]\label{vasiu}Let $G$ be a reductive group over $k$. If $G$ has a non-trivial normal unipotent subgroup scheme then $\Char k=2$ and $G$ has a direct factor isomorphic to $\SO_{2n+1}$. Furthermore, if $G=\SO_{2n+1}$ then $U\cong \alpha_2^{2n}$ is the unique such; and $\Lie(U)$ is a $2n$-dimensional module for $\SO_{2n+1}$ of high weight $\varpi_1$.\end{theorem}

\begin{remark}In the theorem above, the $2n$-dimensional module $L(\varpi_1)$ is obtained as a quotient of the `natural' Weyl module $V(\varpi_1)$ by the radical of its form; see \cite[II.8.21]{Jan03} for a brief discussion.\end{remark}

As is rather well-known (see \cite[2.1]{Vas05}) we have that $\SO_{2n+1}/U\cong\Sp_{2n}$, where $U\cong \alpha_2^{2n}$ is its infinitesimal unipotent normal subgroup. The following is now immediate from the theorem and the fact that $L(\varpi_1)$ is irreducible.

\begin{cor}\label{vascor} Let $G$ be a linear algebraic group over $k$ admitting a reductive subgroup $C$ such that $G=C\RR_u(G)$. Then either the quotient map $q:G\to G/\RR_u(G)$ restricts to an isomorphism on $C$ or $C$ contains a direct factor $H$ isomorphic to $\SO_{2n+1}$ and the image of $H$ under $q$ is isomorphic to $\Sp_{2n}$.\end{cor}

\begin{proof}[Proof of Theorem]
In \cite[Prop.~5.11]{LS12} a subgroup $C\subseteq C_G(e)$ is constructed such that $C_G(e)^\circ_{\red}=C\RR_u(C_G(e))$. One finds that $C$ contains direct factors of type $\SO_{2n+1}$ only if $G$ is $\O(V)$ for some $V$, hence Corollary \ref{vascor} implies $\Lie(C)\cap\Lie(\RR_u(G))$ is trivial when $G=\Sp_{2n}$. 

Hence we assume $G$ is $O(V)$ and $C$ contains a direct factor isomorphic to $\SO_{2r+1}$. The proof of \cite[Prop.~5.11]{LS12} proceeds by finding an orthonormal basis for $V$ and describing explicitly the action of $e$ on $V$. One finds that the action of $e$ on $V$ is constructed as a direct sum of non-isomorphic indecomposable $ke$-modules which are labelled $W(m_i)$ and $W_l(n)$; a basis of these modules and explicit action of $e$ is given in \cite[\S5.1]{LS12}. The multiplicity of the module $W(m_i)$ is labelled $a_i$, thus $W(m_i)^{a_i}$ appears as a direct $ke$-summand of $V$. Furthermore, a certain $1$-dimensional torus $T\subset G$ associated to $e$ is constructed which stabilises each of the indecomposable $ke$-modules above. Then $C$ is constructed as a subgroup of $C_G(T,e)=C_G(T)\cap C_G(e)$. It turns out that the non-zero weight spaces of $T$ on $C_G(e)$ are all of positive weight; denoting the corresponding subgroup by $C_G(e)_{>0}$ we have  $C_G(e)_{>0}\subseteq \RR_u(C_G(e))$. Thus it suffices to show that $\RR_u(C_G(T,e))\cap C=\{1\}$, scheme-theoretically.

We proceed by identifying, for each direct factor $H$ of type $\SO_{2r+1}$ in $C_G(e)$, a $C_G(T,e)$-submodule of $V$ on which $H$ acts faithfully and on which $\RR_u(C_G(T,e))$ acts trivially. This is enough to prove the theorem.

Since $C$ contains a direct factor isomorphic to $\SO_{2r+1}$ we have from \cite[Lem.~5.10]{LS12} that $V$ contains a summand of the form $W_l(n)$ with $2(n-l)\leq m_i\leq 2l-1$. Then following the proof of \emph{loc.~cit.} we obtain an action of $\SO_{2a_i+1}$ on the zero weight space $Z_0$ of the module $Z:=W(m_i)^{(a_i)}\perp W_l(n)$. Given the explicit description of the modules $W(m_i)$ and $W_l(n)$ from \cite[\S5.1]{LS12}, we have that $Z_0$ is non-degenerate of dimension $2a_i+2$. Then the proof of \cite[Lem.~5.10]{LS12} describes $\SO_{2a_i+1}$ as acting on $Z_0$ as the indecomposable module with successive factors being the trivial module $k$, $L(\varpi_1)$ and $k$ again (or $k$, $L(2\varpi_1)$, $k$ if $Y\cong \SO_3=\PGL_2$). Since the natural module for $\SO_{2a_i+1}$ is isomorphic to the unique codimension $1$-submodule of $Z_0$, we have that $\SO_{2a+1}$ acts faithfully on this module. As is well-known, $\SO_{2a_i+1}$ is contained in no parabolic subgroup of $\O_{2a_i+2}$. Hence by the Borel--Tits theorem, the image of $C_G(T,e)$ in $\O_{2a_i+2}$ must be reductive. Thus its unipotent radical $\RR_u(C_G(T,e))$ acts trivially on the faithful $\SO_{2a_i+1}$-module $Z_0$ as required.
\end{proof}

\section{A question}

It is possible for a reductive subgroup $L$ of an algebraic group $G=LU$ to satisfy (\ref{pts}) but not (\ref{lie}). This occurs specifically when $L=\SO_{2n+1}\subset G:=\Sp_{2n}\ltimes V$ where $V$ is the natural module for $\Sp_{2n}$. Nevertheless, $G$ evidently does have a Levi subgroup. In light of this, we raise the following question.

\begin{question}Suppose $G$ is an algebraic group over $k=\bar k$ with unipotent radical $U$, and $L$ is a subgroup which satisfies $G(k)=L(k)U(k)$. Must $G$ have a Levi factor $L'$ such that $G=L'\ltimes U$?\end{question}

{\footnotesize
\bibliographystyle{amsalpha}
\bibliography{bib}}

\providecommand{\bysame}{\leavevmode\hbox to3em{\hrulefill}\thinspace}
\providecommand{\MR}{\relax\ifhmode\unskip\space\fi MR }
\providecommand{\MRhref}[2]{%
  \href{http://www.ams.org/mathscinet-getitem?mr=#1}{#2}
}
\providecommand{\href}[2]{#2}
\begin{thebibliography}{CGP10}

\bibitem[CGP10]{CGP10}
Brian Conrad, Ofer Gabber, and Gopal Prasad, \emph{Pseudo-reductive groups},
  New Mathematical Monographs, vol.~17, Cambridge University Press, Cambridge,
  2010. \MR{2723571 (2011k:20093)}

\bibitem[Jan03]{Jan03}
J.~C. Jantzen, \emph{Representations of algebraic groups}, second ed.,
  Mathematical Surveys and Monographs, vol. 107, American Mathematical Society,
  Providence, RI, 2003. \MR{MR2015057 (2004h:20061)}

\bibitem[Jan04]{Jan04}
\bysame, \emph{Nilpotent orbits in representation theory}, Lie theory, Progr.
  Math., vol. 228, Birkh{\"a}user Boston, Boston, MA, 2004, pp.~1--211.
  \MR{2042689 (2005c:14055)}

\bibitem[LS12]{LS12}
Martin~W. Liebeck and Gary~M. Seitz, \emph{Unipotent and nilpotent classes in
  simple algebraic groups and {L}ie algebras}, Mathematical Surveys and
  Monographs, vol. 180, American Mathematical Society, Providence, RI, 2012.
  \MR{2883501}

\bibitem[McN03]{McN03}
George~J. McNinch, \emph{Faithful representations of {$\rm SL_2$} over
  truncated {W}itt vectors}, J. Algebra \textbf{265} (2003), no.~2, 606--618.
  \MR{1987019}

\bibitem[McN10]{McN10}
\bysame, \emph{Levi decompositions of a linear algebraic group}, Transform.
  Groups \textbf{15} (2010), no.~4, 937--964. \MR{2753264 (2012b:20110)}

\bibitem[McN14]{McN14}
\bysame, \emph{Linearity for actions on vector groups}, J. Algebra \textbf{397}
  (2014), 666--688. \MR{3119244}

\bibitem[Mos56]{Mos56}
G.~D. Mostow, \emph{Fully reducible subgroups of algebraic groups}, Amer. J.
  Math. \textbf{78} (1956), 200--221. \MR{0092928 (19,1181f)}

\bibitem[Ric67]{Ric67}
R.~W. Richardson, Jr., \emph{Conjugacy classes in {L}ie algebras and algebraic
  groups}, Ann. of Math. (2) \textbf{86} (1967), 1--15. \MR{0217079 (36 \#173)}

\bibitem[Spr98]{Spr98}
T.~A. Springer, \emph{Linear algebraic groups}, second ed., Progress in
  Mathematics, vol.~9, Birkh{\"a}user Boston Inc., Boston, MA, 1998.
  \MR{1642713 (99h:20075)}

\bibitem[Ste13]{SteUni}
David~I. Stewart, \emph{On unipotent algebraic {$G$}-groups and 1-cohomology},
  Trans. Amer. Math. Soc. \textbf{365} (2013), no.~12, 6343--6365. \MR{3105754}

\bibitem[Vas05]{Vas05}
A.~Vasiu, \emph{Normal, unipotent subgroup schemes of reductive groups}, C. R.
  Math. Acad. Sci. Paris \textbf{341} (2005), no.~2, 79--84. \MR{2153960
  (2006g:14076)}

\end{thebibliography}

\end{document}